\documentclass[11pt]{article}
\usepackage[margin=1in]{geometry}
\usepackage{amsmath,amssymb,amsthm,mathtools}
\usepackage{comment}
\usepackage{microtype}
\usepackage[colorlinks=true,linkcolor=blue,citecolor=blue,urlcolor=blue]{hyperref}
\newtheorem{theorem}{Theorem}[section]
\newtheorem{definition}[theorem]{Definition}
\newtheorem{lemma}[theorem]{Lemma}
\newtheorem{proposition}[theorem]{Proposition}
\newtheorem{claim}[theorem]{Claim}
\newtheorem{remark}[theorem]{Remark}

\newcommand{\Surf}{\Sigma}
\newcommand{\R}{\mathbb R}
\newcommand{\Diff}{\operatorname{Diff}}
\newcommand{\Imm}{\operatorname{Imm}}
\newcommand{\pp}{\partial}

\title{Immortal Curve Shortening Flow on Analytic Surfaces}
\author{Aprameya Girish Hebbar\thanks{Department of Mathematics, Rutgers University, Piscataway, NJ 08854. Emails: \texttt{ah1531@math.rutgers.edu}, \texttt{gs977@math.rutgers.edu}, \texttt{qs176@math.rutgers.edu}.}
\and Guanhua Shao\footnotemark[1] \and Qi Sun\footnotemark[1]}

\date{}

\begin{document}

\maketitle

\begin{abstract}
We show that on any closed orientable real-analytic Riemannian surface, any smoothly immersed closed immortal curve-shortening flow converges to a unique non-constant closed geodesic. 
\end{abstract}

\section{Introduction}
Let $(\Surf,g)$ be a closed (compact without boundary) orientable real-analytic Riemannian manifold of dimension two and let $\Gamma_0:S^1\rightarrow\Sigma$ be a smoothly immersed curve. We denote by $\Gamma\colon S^1\times[0,T)\to \Surf$, 
\begin{equation}
\label{eqn:CSF}
(\partial_t\Gamma)^\perp=\boldsymbol{k},\quad T\in(0,+\infty]
\end{equation}
the maximal smoothly immersed curve-shortening flow (CSF) starting with the initial curve $\Gamma(\cdot,0)=\Gamma_0$, where $\boldsymbol{k}$ is the (geodesic) curvature vector. In this paper, we focus on the case $T=+\infty$ and we call such a solution \emph{immortal}.

\subsection{Setup}
Following \cite{Angenent2005}, we denote by
\[
  \Imm(S^1,\Sigma)
  :=
  \left\{
    \gamma\in C^\infty(S^1,\Sigma)
    \;\middle|\;
    \partial_\theta\gamma(\theta)\neq0
    \text{ for every }\theta\in S^1
  \right\}
\] the space of parameterized smoothly immersed curves, by
\[
\Diff^+(S^1)
:=
\left\{
\varphi\colon S^1\to S^1
\;\middle|\;
\varphi \text{ is a smooth orientation-preserving diffeomorphism}
\right\}
\]
the diffeomorphism group and we consider the space of unparametrized smoothly immersed curves with orientation,
\[
  \Omega
  :=
  \Imm(S^1,\Sigma)/\Diff^+(S^1).
\]
We write elements of $\Omega$ as $[\gamma]$ for $\gamma\in \operatorname{Imm}(S^1,\Sigma)$ and equip $\Omega$ with the quotient topology induced by the $C^\infty$ topology on $\Imm(S^1,\Sigma)$ as in \cite{CerveraMascaroMichor1991}; see also \cite[Section~5]{AryanLee2026}. It follows that $[\gamma_i]\to [\gamma]$ in $\Omega$ if and only if there exists $\varphi_i\in \operatorname{Diff}^+(S^1)$ such that $\gamma_i\circ \varphi_i\to \gamma$ in $C^\infty(S^1,\Surf)$. 

\subsection{Main result}

\begin{theorem}
\label{thm:main}
For any immortal curve-shortening flow $\Gamma$ on $(\Surf,g)$, there exists a unique non-constant closed geodesic $[\gamma_*]\in\Omega$ such that $[\Gamma(\cdot,t)]\to [\gamma_*]$ in $\Omega$ as $t\to \infty$. 

Moreover, if the initial curve $\Gamma_0$ is embedded, then $\gamma_*$ is an embedded closed geodesic of multiplicity one. 
\end{theorem}
Equivalently, there exists some $\varphi_t\in \operatorname{Diff}^+(S^1)$ such that
\[
  \Gamma(\varphi_t(\cdot),t)\to \gamma_*
  \qquad\text{in }C^\infty(S^1,\Surf),
\]
as $t\to\infty$.  In particular, if $t_i\to \infty$, $\psi_i\in \operatorname{Diff}^+(S^1)$, and $\Gamma(\psi_{i}(\cdot), t_i) \to \bar \gamma$ for some immersion $\bar \gamma$, then there exists  $h \in \Diff^+(S^1)$ such that $\bar \gamma = \gamma_* \circ h$. 

\subsection{Background}
For finite-time singularities of planar CSF, the well-known Gage-Hamilton-Grayson theorem \cite{GageHamilton1986,Grayson1987} says that any embedded closed curve shrinks to a round point. For singularities of embedded CSF on surfaces, see \cite{Zhu1998,Angenent1990,Angenent1991,Oaks1994,JohnsonMuraleetharan2010,Edelen2015}; see also \cite{sun2024curve,sun2025huisken,sun2025singularities} for recent extensions to CSF with convex projections in higher codimensions.

For embedded immortal CSF on surfaces, based on classical results by Gage \cite{Gage1990} and Grayson \cite{Grayson1989}, the geodesic curvatures tend uniformly to zero, and the evolving curves converge smoothly subsequentially to closed geodesics. Those results extend to the immersed setting; see  \cite{Grayson1989} and \cite[Lemma~3.1]{Angenent2005}. 

Uniqueness of geodesic limits is known in some special cases; see \cite{Gage1990,Angenent1999,bryan2026sharp}. However, uniqueness can fail when the ambient metric is merely smooth: see recent work by Aryan and Lee \cite{AryanLee2026}. 

Grayson suggested that real analyticity of the ambient metric should restore uniqueness \cite[p.~74]{Grayson1989}; Theorem~\ref{thm:main} confirms this expectation for closed orientable surfaces and extends the conclusion to immersed curves. 

\subsection{Outline of the proof}
Our argument mainly relies on the \L ojasiewicz-Simon gradient inequality,
originating in Simon's work on evolution equations
\cite{Simon1983} and used here through the abstract Banach-space
version of Feehan and Maridakis \cite{FeehanMaridakis2020}.  
Closely related arguments appear for elastic flows of immersed closed curves in analytic
manifolds \cite{Pozzetta2022} and for motion by curvature of planar
networks near a minimal network \cite{PludaPozzetta2024}; see also related works on uniqueness of tangent flows \cite{schulze2014uniqueness,colding2015uniqueness,chodosh2021uniqueness,zhu2020ojasiewicz,lee2024uniqueness}. 

In Section~\ref{sec:prelims}, we introduce normal graph coordinates near one fixed geodesic limit. In Section~\ref{sec:analytic}, we show that the length functional is
real analytic and recall the \L ojasiewicz-Simon gradient inequality. 
In Section~\ref{sec:trapping}, we use this inequality to show that the curve becomes trapped in any prescribed neighborhood of the limit for all sufficiently large times, completing the proof of Theorem~\ref{thm:main}. 

\noindent \textbf{Acknowledgments. }A.G.H. thanks Nata\v{s}a \v{S}e\v{s}um for her continued support. G.S. thanks Daniel Ketover for his continued support and encouragement. Q.S. thanks Sigurd Angenent for some conversations on CSF on manifolds. 

\noindent \textbf{Statement on the use of AI. }
OpenAI’s GPT-5.6 Sol was used in preparing the first draft, but the final mathematical arguments were developed, written, and verified by the authors. The authors take full responsibility for the correctness and content of the paper.

\section{Preliminaries}
\label{sec:prelims}
We regard $S^1$ as the standard unit circle $\R/2\pi\mathbb Z$, with angular coordinate $\theta\in[0,2\pi)$ and measure $d\theta$.  
For an integer $m\geq0$, we denote by
\[
 H^m(S^1)
 :=
 \left\{
 u\in L^2(S^1): \text{weak derivatives }
 \partial_\theta^j u\in L^2(S^1)
 \text{ for }0\leq j\leq m
 \right\},
\]
the Sobolev space with norm and inner product 
\[
 \|u\|_{H^m(S^1)}^2
 :=
 \sum_{j=0}^m\int_{S^1}|\partial_\theta^j u|^2\,d\theta\text{ and }\langle u,v\rangle_{H^m}
 :=
 \sum_{j=0}^m\int_{S^1}
 \partial_\theta^j u\,\partial_\theta^j v\,d\theta.
\]
We use the standard one-dimensional Sobolev embeddings. In particular, $H^2(S^1)\hookrightarrow C^1(S^1)$ continuously and $H^2(S^1)\Subset C^1(S^1)$ compactly. 

The next lemma follows from previous work of Gage \cite{Gage1990}, Grayson \cite{Grayson1989} and Angenent \cite{Angenent2005}.
\begin{lemma}
\label{lem:global}
Let \(\Gamma\colon S^1\times[0,\infty)\to \Surf\) be an immortal smooth CSF through immersions. Write $\Gamma_t:=\Gamma(\cdot,t)$ for $t\in [0,\infty)$. Then the length $L(t)=L(\Gamma_t)$ decreases to a positive limit $L_\infty$, and
\begin{equation}
\label{eqn:curv-to-zero}
\sup_{\Gamma_t}|k|\to  0 \qquad \text{as }t\to  \infty.
\end{equation}
Moreover, every sequence $t_i\to\infty$ has a subsequence for which $[\Gamma_{t_i}]$ converges in $\Omega$ to a nonconstant immersed closed geodesic in $\Omega$.
If $\Gamma_0$ is embedded, then every such limit is embedded and has multiplicity one. 
\end{lemma}
\begin{remark}
The limit obtained may depend on the chosen sequence.
For general immersed CSF $\Gamma_t$, limits obtained in the above lemma can fail to be embedded or have multiplicity greater than $1$.
\end{remark}
Let \(\Gamma\colon S^1\times[0,\infty)\to \Surf\) be a smooth curve-shortening flow through immersions. From now on, we fix a geodesic $\gamma_*\colon S^1\to\Surf$ obtained as a subsequential
limit using Lemma~\ref{lem:global} and parametrize it with constant speed so that 
\[
 |\partial_\theta\gamma_*|_g=\frac{\ell}{2\pi},
 \qquad
 \ell=L(\gamma_*)=L_\infty.
\]
Let \(\mathbf T_*:=\frac{2\pi}{\ell}\,\partial_\theta\gamma_*\) be its oriented unit tangent. We denote by $ds$ the arclength measure along $\gamma_*$, pulled back to $S^1$, and by $\partial_s$ the corresponding arclength derivative. Thus
\begin{equation}
\label{eq:theta-to-s}
 ds=\frac{\ell}{2\pi}\,d\theta,
 \qquad
 \partial_s=\frac{2\pi}{\ell}\,\partial_\theta.
\end{equation}
Whenever a function below is written in the variable $s$, we mean its expression in the arclength coordinate $s=(\ell/2\pi)\theta\pmod\ell$. The Sobolev spaces defined above are
unchanged if one uses $d\theta$ and $ds$ as the resulting norms are equivalent, with constants depending only on $\ell$.

Fix an orientation of $\Surf$ and let $J:T\Sigma \to T\Sigma$ be the complex structure on $\Sigma$. Set $\nu_*:=J\mathbf T_*$, which is a unit normal along $\gamma_*$.  For $\varepsilon>0$ small, define
\[
 F\colon S^1\times(-\varepsilon,\varepsilon)\to\Surf,
 \qquad
 F(\theta,r):=\exp_{\gamma_*(\theta)}\!\bigl(r\nu_*(\theta)\bigr), 
\]
so that $F$ is a local diffeomorphism. Since $\gamma_*$ may be only immersed, $F$ need not be injective. Let $\mathcal{N}_\varepsilon$ denote the image of $F$. For $u\in H^2(S^1)$ satisfying $\|u\|_{C^0}<\varepsilon,$ we denote by
\[
    \gamma_u\colon S^1\to \mathcal N_\varepsilon,
    \qquad
    \gamma_u(\theta)=F(\theta,u(\theta)), \theta \in S^1,
\]
the corresponding normal graph over $\gamma_*$. Since $H^2(S^1)\hookrightarrow C^1(S^1)$, the map $\gamma_u$ is a $C^1$ immersion into $\mathcal{N}_\varepsilon$. Using Sobolev embedding, choose $\delta>0$ sufficiently small (depending on $\varepsilon$) that
\begin{equation}
\label{eq: C^1 bound of u}
\|u\|_{H^2}\leq\delta
\quad\Rightarrow \quad
\|u\|_{C^0}<\frac{\varepsilon}{2}
\quad\text{and}\quad
\|\pp_s u\|_{C^0}<\frac{\varepsilon}{2},
\end{equation}
Fix the closed ball
\begin{equation}
\label{eq:U-delta}
    \mathcal U_{\delta}
    :=
    \left\{
        u\in H^2(S^1):
        \|u\|_{H^2}\leq \delta
    \right\}.   
\end{equation}
When using the arclength coordinate $s=(\ell/2\pi)\theta$, we use the same
symbols $F,\gamma_*,\nu_*$, and $u$ for their reparametrizations. Using Gauss' lemma, the pullback metric on $S^1\times (-\varepsilon/2,\varepsilon/2)$ can be written as 
\begin{equation}
\label{eq:Fermi-metric}
  F^*g=dr^2+A(s,r)^2\,ds^2,
\end{equation}
where \(A(s,r)=|\partial_sF(s,r)|_g>0.\) Since $g$ is real analytic, $A$ is real analytic in its arguments. Since $\gamma_*$ is a geodesic and $s$ is arclength along it, we have 
\begin{equation} \label{eq: Gauss Lemma}
        A(s,0)=1,
    \qquad
    (\pp_r A)(s,0)=0.
\end{equation}
Moreover, the Jacobi equation gives
\begin{equation}
\label{eq:FermiJacobi}
(\pp_r^2A)(s,0)=-K_g(\gamma_*(s)).
\end{equation}
From \eqref{eq: Gauss Lemma}, by shrinking $\varepsilon > 0$ if necessary, we may assume 
\begin{equation} \label{eq: bound on A}
    \frac{3}{2} > A(s, u(s)) > \frac{1}{2} \quad \text{on }S^1
\end{equation}
for any $u \in \mathcal U_{\delta}$. For a graph $u$ we define
\begin{equation}
\label{defn:Q-u}
Q[u](s):=\sqrt{A(s,u)^2+|\pp_s u|^2}.
\end{equation}
Then we have
\begin{equation} \label{eq: bound on Q}
    \frac{1}{2} \leq Q[u](s) \leq \sqrt{\frac{9}{4} + \varepsilon^2} \leq 2
\end{equation}
for any $s \in S^1$ and $u \in \mathcal U_{\delta}$. 

As an application, we have the following regularity. 
\begin{lemma}
\label{lem:uniform}
For every integer $m>2$, there exist constants $B_m<\infty$ (depending on $\Gamma$) such that if $t\geq 1$, $u^t\in \mathcal{U}_\delta\cap C^\infty(S^1)$ such that $[\Gamma_t]=[\gamma_{u^t}]$ in $\Omega$, then 
\[
    \|u^t\|_{H^m(S^1)}\leq B_m .
\]
\end{lemma}
\begin{proof}
Suppose the conclusion fails for some \(m>2\).  
Then, for each \(i\), there are \(t_i\geq 1\) and \(u_i:=u^{t_i}\in\mathcal U_\delta\cap C^\infty(S^1)\) such that \([\Gamma_{t_i}]=[\gamma_{u_i}]\), \(\|u_i\|_{H^m}>i.\) 
If $t_\infty:=\sup_i t_i<\infty$, then we obtain a contradiction since the CSF is smooth on $[\frac{1}{2},t_\infty+1]$. 
Thus, we may assume that $t_i\to \infty$. 
After passing to a subsequence, the \(H^2\)-bound and the compact embedding \(H^2(S^1)\Subset C^1(S^1)\) give $u_i\rightharpoonup u_\infty$ in $H^2(S^1)$, $u_i\to u_\infty$ in $C^1(S^1)$ for some $u_\infty\in\mathcal U_\delta.$ 
Lemma~\ref{lem:global} gives, after passing to a further subsequence, an immersed closed geodesic \(\bar\gamma\) such that $[\gamma_{u_i}]=[\Gamma_{t_i}]\to [\bar \gamma]$ in $\Omega$. 
Since $u_i\to u_\infty$ in $C^{1}(S^1)$, Lemma~\ref{lem:graph-reparam-compactness} shows that $u_\infty\in\mathcal U_\delta\cap C^\infty(S^1)$ and after passing to a further subsequence, $u_i\to u_\infty$ in $C^\infty(S^1)$. In particular, $\sup_i\|u_i\|_{H^m}<\infty,$ contradicting \(\|u_i\|_{H^m}\to\infty.\) This proves the lemma. 
\end{proof}

\section{The length functional}
\label{sec:analytic}
We continue to use the same setup as in the previous section. For $u\in\mathring{\mathcal U}_{\delta}$, we define the \emph{length functional} by
\[
    E(u)
    :=
    L(\gamma_u)
    =
    \int_{S^1}Q[u](s)\,ds,
\]
where $\mathring{\mathcal{U}}_\delta$ is the interior of $\mathcal{U}_\delta$ defined in \eqref{eq:U-delta} and $Q$ is defined by \eqref{defn:Q-u}. 

We denote by $L^2(S^1,ds)$ the Hilbert space, with inner product $\langle f,h\rangle_{L^2(ds)}
:=
\int_{S^1}fh\,ds$. 
\begin{lemma}
\label{lem:LS-preparation}
After decreasing $\delta>0$ if necessary, the following hold. 
\begin{enumerate}
    \item The length functional \(E\colon\mathring{\mathcal U}_{\delta}\to \mathbb R\) is real analytic.
    \item The functional $E$ has an $L^2(S^1,ds)$ gradient \(M\colon\mathring{\mathcal U}_{\delta}
        \to
        L^2(S^1,ds)\) given by \[
        M(u)
        =
        \frac{A(s,u)(\pp_r A)(s,u)}{Q[u]}
        -
        \partial_s\left(\frac{\pp_s u}{Q[u]}\right).
    \]
Moreover, \(M\colon\mathring{\mathcal U}_{\delta}
        \to 
        L^2(S^1,ds)\) is real analytic. 
    \item We have \(M(0)=0\), and
    \begin{equation}
    \label{eq:jacobi}
        M'(0)v
        =
        -\pp_s^2 v
        -
        K_g(\gamma_*(s))v.
    \end{equation}
As an operator \(M'(0)\colon H^2(S^1)
        \to 
        L^2(S^1,ds)\), it is Fredholm of index zero. 
\end{enumerate}
\end{lemma}

\begin{proof}
Since $A$ is real analytic on a neighborhood of $S^1\times\{0\}$, after decreasing $\delta > 0$ the maps
\(u\mapsto A(\,\cdot\,,u(\,\cdot\,)),
    u\mapsto (\pp_r A)(\,\cdot\,,u(\,\cdot\,))\) are real analytic on $\mathring{\mathcal U}_{\delta}$ with values in $H^2(S^1)$. The maps $w\mapsto \sqrt{w}$ and $w\mapsto w^{-1/2}$ are real analytic from $\left\{
w\in H^1(S^1):
\min_{S^1}w>0
\right\}$ to $H^1(S^1)$. Combined with \eqref{eq: bound on Q}, it follows that $u\mapsto Q[u]$ is a real analytic map from $\mathring{\mathcal U}_{\delta}$ to $H^1(S^1)$. Consequently, $E$ is a real analytic map from $\mathring{\mathcal U}_{\delta}$ to $\mathbb R$. This proves the first assertion.

The $L^2(S^1,ds)$ gradient $M$ of $E$ is defined by 
\[
    DE(u)[v]:=\left.\frac{d}{dt}\right|_{t=0}E(u+tv)
    =
    \langle M(u),v\rangle_{L^2(ds)}
    \qquad
    \text{for all }v\in H^2(S^1).
\]
For $u,v\in H^2(S^1)$,
\[
    DQ[u][v]
    =
    \frac{
        A(s,u)(\pp_r A)(s,u)v+\pp_s u\pp_s v
    }{Q[u]}.
\]
Hence
\[
\begin{aligned}
    DE(u)[v]
    &=
    \int_{S^1}
    \frac{
        A(s,u)(\pp_r A)(s,u)v+\pp_s u\,\pp_s v
    }{Q[u]}
    \,ds 
    =
    \int_{S^1}
    \frac{A(s,u)(\pp_r A)(s,u)}{Q[u]}\,v\,ds
    +
    \int_{S^1}
    \frac{\pp_s u}{Q[u]}\pp_s v\,ds.
\end{aligned}
\]
Integration by parts gives 
\[
    DE(u)[v]
    =
    \int_{S^1}
    \left[
        \frac{A(s,u)(\pp_r A)(s,u)}{Q[u]}
        -
        \partial_s\left(\frac{\pp_s u}{Q[u]}\right)
    \right]v\,ds.
\]
Thus
\begin{equation} \label{eq:M-formula}
    M(u)
    =
    \frac{A(s,u)(\pp_r A)(s,u)}{Q[u]}
    -
    \partial_s\left(\frac{\pp_s u}{Q[u]}\right).
\end{equation}
Since $Q[u]$ is uniformly bounded below and $H^1(S^1)$ is an algebra, it follows that $M\colon\mathring{\mathcal U}_\delta\to L^2$  is real analytic. This proves the second assertion.

Using $A(s,0)=1$ and $(\pp_r A)(s,0)=0$, we obtain \(M(0)=0\).  We now
linearize $M$ at zero. For the first term, using \eqref{eq:FermiJacobi} and $Q[0]=1$,
\[
    D\left[
        \frac{A(s,u)(\pp_r A)(s,u)}{Q[u]}
    \right]_{u=0}[v]
    =
    (\pp_r^2A)(s,0)v=-K_g(\gamma_*(s))v.
\]
For the second term,
\[
    D\left[
        -\partial_s\left(\frac{\pp_s u}{Q[u]}\right)
    \right]_{u=0}[v]
    =
    -\pp_s^2 v.
\]
This proves \eqref{eq:jacobi}. Now the third assertion follows from the standard elliptic PDE theory (for instance see \cite[\S 6.2.3]{Evans1998}). This completes the proof. 
\end{proof}
Since the metric $g$ and the functional $E$ are real analytic, we now recall the \L ojasiewicz-Simon gradient inequality. 
\begin{proposition}
\label{prop:LS-inequality} 
By shrinking $\delta$ further, there exist constants \( C>0,\beta\in(0,1/2]\) such that if $u\in \mathring{\mathcal{U}}_\delta$, then
\begin{equation}
\label{eq:LS}
    |E(u)-E(0)|^{1-\beta}
    \leq
    C\|M(u)\|_{L^2(S^1,ds)}.
\end{equation}
\end{proposition}
\begin{proof}
We have the continuous embeddings
\(H^2(S^1)
    \hookrightarrow
    L^2(S^1,ds)
    \hookrightarrow
    H^2(S^1)^*,\)
where the second embedding is induced by the $L^2(ds)$ pairing and the
resulting embedding $H^2(S^1)\hookrightarrow H^2(S^1)^*$ is definite in the
sense of \cite{FeehanMaridakis2020}.  Together with
Lemma~\ref{lem:LS-preparation}, we may apply
\cite[Theorem~2]{FeehanMaridakis2020} and obtain \eqref{eq:LS} (see also \cite[Theorem~3]{Simon1983}). 
\end{proof}

\section{Trapping and convergence}
\label{sec:trapping}
We continue to use the same setup as in the previous sections. Recall that $\Gamma$ is the original flow fixed in Section~\ref{sec:prelims}. 
\begin{definition}
Let $I\subset[0,\infty)$ be an interval.  We say that $\Gamma$ is a \emph{graphical solution} to CSF on $I$ if there exists a smooth map $u\colon S^1\times I\to(-\varepsilon,\varepsilon)$ such that $u(\cdot,t)\in\mathring{\mathcal U}_\delta$ and for every \(t\in I\) 
$[\Gamma_t]=[\gamma_{u(\cdot,t)}]$ in $\Omega$. 
\end{definition}
Any function $u$ satisfying the above equation will be called a normal graph function associated with $\Gamma$ on $I$. In general, such a function may be non-unique. The next lemma computes the parabolic equation for a graphical solution to CSF. 
\begin{lemma}
\label{lem:graph-form-csf}
Let $I\subset[0,\infty)$ and suppose that $\Gamma$ is a graphical solution to CSF on $I$ with a normal graph function $u$. Then $u$ satisfies 
\begin{equation}
\label{eq:graph-csf}
\partial_t u=-b(u)M(u),
\qquad
b(u):=\frac{Q[u]}{A(s,u)^2},
\end{equation}
where $Q$ is defined by \eqref{defn:Q-u}. Along this evolution,
\begin{equation}
\label{eq:graph-energy-identity}
-\frac{d}{dt}E(u(\cdot, t))
  =\int_{S^1} b(u)M(u)^2\,ds
  =\int_{\Gamma_t} k^2\,d\lambda_t.
\end{equation}
Here $d\lambda_t$ is arclength on $\Gamma_t$.  
Moreover, for $u$ in a sufficiently small $H^2$ ball, \eqref{eq:graph-csf} is uniformly parabolic. 
Conversely, if \(u\in C^\infty(S^1\times I)\), with \(u(\cdot,t)\in\mathring{\mathcal U}_\delta\), satisfies \eqref{eq:graph-csf}, then $t\mapsto \gamma_{u(\cdot,t)}$ is a solution to CSF. 
\end{lemma}
\begin{remark}
Note that the last integral in \eqref{eq:graph-energy-identity} is over $\Gamma_t$ with its induced arclength, so it counts every iteration of $\Gamma_t$.
\end{remark}

\begin{proof}[Proof of Lemma \ref{lem:graph-form-csf}]
Since $\Gamma_t$ is graphical, there exists $u(\cdot,t):S^1\to (-\varepsilon,\varepsilon)$ such that $[\Gamma_t]=[\gamma_{u(\cdot,t)}]$ in $\Omega$.  Set $X(\cdot,t):=F(\cdot,u(\cdot,t))$ so that $(\partial_tX)^\perp=\boldsymbol{k}_t$, where $\boldsymbol{k}_t$ is the curvature vector of $X(\cdot,t)$.  
Write $A=A(s,u)$. We have 
\[
\partial_s X=\pp_s F+\pp_s u\,\pp_rF, \quad \pp_tX = \pp_t u\,\pp_rF.
\]
The arclength measure $d\lambda_{t}$ and the oriented unit normal $N_t$ on $\Gamma_t$ is given by
\[
d\lambda_t=Q\,ds, \quad N_t=\frac{A}{Q}\pp_r F-\frac{\pp_s u}{AQ}\pp_s F.
\]
Since \(\langle \pp_t X,N_t\rangle=\frac{A}{Q}\pp_t u\), the CSF equation gives
\(\pp_t u=(Q/A)k_t.\) 
The first variation formula then gives
\[
DE(u)[v]
=
-\int_{\Gamma_t}k\langle v\, \pp_r F,N_t\rangle\,d\lambda_t
=
-\int_0^\ell A\,k\, v\,ds.
\]
Hence $M(u)=-Ak$, which implies
\[
\pp_t u=-\frac{Q}{A^2}M(u),
\]
showing \eqref{eq:graph-csf}. \eqref{eq:graph-energy-identity} now follows.  By
\eqref{eq:M-formula},
\[
\begin{split}
M(u)
& =
\frac{A(s,u)(\pp_r A)(s,u)}{Q}
-
\partial_s\left(\frac{\pp_s u}{Q}\right) 
=
\frac{A\pp_r A}{Q}
-\frac{\pp_s^2 u}{Q}
+\frac{\pp_s u}{Q^3}
 \left[A\bigl(\pp_s A+\pp_r A\, \pp_s u\bigr)+(\pp_s u)(\pp_s^2 u)\right] \\
& =
-\frac{A^2}{Q^3}\pp_s^2 u
+\frac{A\pp_r A}{Q}
+\frac{A\,\pp_s u\bigl(\pp_s A+\pp_r A \pp_s u\bigr)}{Q^3} 
=
-\frac{A^2}{Q^3}\pp_s^2 u
+\text{lower-order terms}.
\end{split}
\]
Consequently,
\[
\pp_t u
=
\frac{1}{Q^2}\pp_s^2 u
-\frac{\pp_r A}{A}
-\frac{\pp_s u(\pp_s A+\pp_r A\pp_s u)}{AQ^2}.
\]
Since $H^2(S^1)\hookrightarrow C^1(S^1)$, the coefficient $Q^{-2}$ is bounded above and bounded away from zero on a sufficiently small $H^2$ ball. Thus \eqref{eq:graph-csf} is uniformly parabolic. 

For the final statement, suppose that $u(\cdot,t)\in\mathring{\mathcal U}_\delta$ satisfies  \eqref{eq:graph-csf}. Setting $X(s,t):=F(s,u(s,t))$, we have $\pp_t X=\pp_t u\,\pp_r F=-\frac{Q}{A^2}M(u)\pp_r F$. Since $M(u)=-Ak_{\gamma_u}$ and $N_{\gamma_u}=\frac{A}{Q}\pp_r F-\frac{\pp_s u}{AQ}\pp_s F$, we obtain 
$$\langle \pp_t X,N_{\gamma_u}\rangle =-\frac{Q}{A^2}(-Ak_{\gamma_u})\langle \pp_r F,N_{\gamma_u}\rangle =k_{\gamma_u},$$
showing \eqref{eqn:CSF}. This completes the proof. 
\end{proof}

\begin{definition}
On any interval on which the flow $\Gamma$ is a graphical solution to CSF with graph function $u$, we define the difference of the length functional
\[
    e(t):=L(\Gamma_t)-L_\infty= E(u(\cdot, t)) - E(0).
\]
By our choice of $\gamma_*$, $e(t)$ is nonnegative and nonincreasing, and $e(t)\to0$ as $t\to\infty$.
\end{definition}
 
Note that by \eqref{eq: bound on A} and \eqref{eq: bound on Q}, there are universal constants $b_+, b_- > 0$ such that \(b_-\leq b(u)(s)\leq b_+\) for every $u\in\mathring{\mathcal U}_{\delta}$. 
\begin{lemma}
\label{lem:pathlength}
There exists $C_0>0$ such that the following holds. Suppose that $\Gamma$ is
a graphical solution to CSF on $[a,b]$ with normal graph function $u$ with $\|u(\cdot, t)\|_{H^2}<\delta$ for every $t\in [a,b]$. Then
\begin{equation}
\label{eq:pathlength}
    \int_a^b\|\pp_t u( \cdot, t)\|_{L^2(S^1,ds)}\,dt
    \leq
    C_0e(a)^\beta.
\end{equation}
\end{lemma}

\begin{proof}
Let $C,\beta$ be the constants given by Proposition~\ref{prop:LS-inequality}. 
Suppose first that $e(t)>0$ for $t\in [a,b]$. Using \eqref{eq:graph-energy-identity}, we compute for $t\in [a,b]$, 
\[
\begin{aligned}
    -\frac{d}{dt}e(t)^\beta
    &=
    \beta e(t)^{\beta-1}
    \int_{S^1}b(u)M(u)^2\,ds.
\end{aligned}
\]
Proposition~\ref{prop:LS-inequality} gives
\[
    e(t)^{1-\beta}
    \leq
    C\|M(u(\cdot, t))\|_{L^2}.
\]
In particular, $M(u(\cdot,t))\neq0$. Consequently,
\[
\begin{aligned}
    -\frac{d}{dt}e(t)^\beta\geq
    \frac{\beta}{C\|M(u(\cdot, t))\|_{L^2}}
    \int_{S^1}b(u)M(u)^2\,ds \geq
    \frac{\beta b_-}{C}
    \|M(u(\cdot, t))\|_{L^2}.
\end{aligned}
\]
Since
\(\pp_t u=-b(u)M(u)\)
and $b(u)\leq b_+$, we also have
\[
    \|\pp_t u(\cdot, t)\|_{L^2}
    \leq
    b_+\|M(u(\cdot, t))\|_{L^2}.
\]
Thus
\[
    -\frac{d}{dt}e(t)^\beta
    \geq
    \frac{\beta b_-}{Cb_+}
    \|\pp_t u(\cdot, t)\|_{L^2}.
\]
Integrating from $a$ to $b$ yields
\[
    \int_a^b\|\pp_t u(\cdot, t)\|_{L^2}\,dt
    \leq
    \frac{Cb_+}{\beta b_-}
    \bigl(e(a)^\beta-e(b)^\beta\bigr)
    \leq
    \frac{Cb_+}{\beta b_-}e(a)^\beta.
\]
Hence \eqref{eq:pathlength} holds with
\(C_0:=\frac{Cb_+}{\beta b_-}.\)

Suppose now that $e$ vanishes at some time in $[a,b]$, and define
$t_0:=\inf\{t\in[a,b]:e(t)=0\}$. Since length is nonincreasing and tends
to $L_\infty$, one has $L(t)=L_\infty$ for all $t\geq t_0$.
The first variation of length then gives
\[
    \int_{\Gamma_t}k^2\,d\lambda_t=0
    \qquad\text{for all }t\geq t_0,
\]
so the flow is stationary after time $t_0$.  
Applying the preceding argument on $[a,t_0)$ proves the same estimate. 
\end{proof}

\begin{proposition}
\label{prop:trapping}
Suppose $t_j\to\infty$ and $[\Gamma_{t_j}]$ converges to $[\gamma_*]$ in $\Omega$.  For every
$\rho\in(0,\delta)$, there exists $j_0$ sufficiently large such that $\Gamma$ is a graphical solution of CSF on $[t_{j_0},\infty)$  and its normal graph function $u$ satisfies
\[
    \|u(\cdot,t)\|_{H^2}<\rho<\delta
    \qquad\text{for all }t\geq t_{j_0}.
\]

\end{proposition}

\begin{proof}
Since $[\Gamma_{t_j}]\to [\gamma_*]$ in $\Omega$, it follows that $[\Gamma_{t_j}]=[\gamma_{u_j}]$ in $\Omega$ for some $u_j\in C^\infty(S^1)$ with $u_j\to0$ in $C^\infty(S^1)$. Fix $m>2$ and let $B_m$ be the positive constant given by Lemma~\ref{lem:uniform}. 
Let $C_0$ be the constant in Lemma~\ref{lem:pathlength}. The interpolation inequality gives a constant $C_m$ such that
\begin{equation}
\label{eq:SobolevInterpolation}
\|w\|_{H^2}
    \leq
    C_m\|w\|_{L^2}^{1-2/m}\|w\|_{H^m}^{2/m},
\end{equation}
for any $w \in H^m(S^1)$. Set
\[
    \eta
    :=
    \left(
        \frac{\rho}
        {2C_m(2B_m)^{2/m}}
    \right)^{m/(m-2)}
    >0.
\]
Note that $e(t_j)\to0$. Choose $j_0$ sufficiently large and set $a:=t_{j_0}$ so that
\[
    a\geq 1,
    \qquad
    \|u_{j_0}\|_{H^2}<\frac{\rho}{4},
    \qquad
    C_0e(a)^\beta<\eta.
\]
Let $\mathcal{T}$ denote the set of all $T\geq a$ such that $\Gamma$ is a graphical solution of CSF on $[a,T]$ with a normal graph function $v_T(\cdot,t)$ such that $v_T(\cdot,a)=u_{j_0}$ and for all $t\in [a,T]$,  
\[
  \|v_T(\cdot,t)\|_{H^2}<\rho.
\]
Let $T_*:=\sup \mathcal{T}$. By uniqueness for CSF starting with smooth initial data \cite[Theorem~3.1]{Angenent1990} and short-time existence and uniqueness for the uniformly parabolic equation \eqref{eq:graph-csf}, it follows that $T_*>a$. Moreover, there exists a smooth function \(u\colon S^1\times[a,T_*)\to \mathbb R\) such that $u(\cdot,a)=u_{j_0},$ $\|u(\cdot,t)\|_{H^2}<\rho,$ and $[\Gamma_t]=[\gamma_{u(\cdot,t)}]$ for all $t\in [a,T_*)$. 

Suppose, for contradiction, that \(T_*<\infty\).
\begin{claim}
\label{claim:trapping}
For every $T_1\in(a,T_*)$,
\[
    \|u(\cdot, T_1)\|_{H^2}<\frac{3\rho}{4}.
\]
\end{claim}
\begin{proof}
Assume, by contradiction, that there exists $T_1\in(a,T_*)$ such that
\(\|u(\cdot, T_1)\|_{H^2}\geq\frac{3\rho}{4}.\)
Set $w:=u(\cdot, T_1)-u(\cdot, a).$
Since \(\|u(\cdot, a)\|_{H^2}<\frac{\rho}{4},\) the triangle inequality gives
\begin{equation} \label{eq:lower-bound-w}
\begin{aligned}
    \|w\|_{H^2}
    &\geq
    \|u(\cdot, T_1)\|_{H^2}
    -
    \|u(\cdot, a)\|_{H^2} >
    \frac{3\rho}{4}-\frac{\rho}{4}
    =
    \frac{\rho}{2}.
\end{aligned}   
\end{equation}
By the definition of $T_*$, \(\|u(\cdot,t)\|_{H^2}<\rho<\delta\) for all
$t\in[a,T_1]$.
Lemma~\ref{lem:uniform} yields  $\|w\|_{H^m}\leq 2B_m.$ 
Combining \eqref{eq:SobolevInterpolation} with \eqref{eq:lower-bound-w}, we have
\[
    \frac{\rho}{2}
    <
    C_m
    \|w\|_{L^2}^{1-2/m}
    (2B_m)^{2/m}.
\]
By the choice of $\eta$, it follows that $\|w\|_{L^2}>\eta.$ We apply Lemma~\ref{lem:pathlength} and obtain
\[
\begin{aligned}
    \|w\|_{L^2}=\|u(\cdot, T_1)-u(\cdot, a)\|_{L^2} \leq \int_a^{T_1}\|\pp_t u( \cdot, t)\|_{L^2}\,dt \leq  C_0e(a)^\beta<\eta,
\end{aligned}
\]
a contradiction. Hence $\|u(\cdot, T_1)\|_{H^2}<\frac{3\rho}{4}.$ 
\end{proof}
Now we use Claim~\ref{claim:trapping} to show that the graphical solution $u(\cdot, t)$ can be smoothly extended over $T_*$, arriving at the contradiction we want. Combining Claim~\ref{claim:trapping} and Lemma~\ref{lem:uniform} shows that for all $l\geq 3$, 
\[
    \sup_{a\leq t<T_*}\|u(\cdot,t)\|_{H^l}
    \leq B_l.
\]
Since $Q^{-2}$ and $A$ stays bounded above and away from zero by \eqref{eq: bound on Q} and \eqref{eq: bound on A}, and the coefficients of \eqref{eq:graph-csf}  are smooth, this gives a constant $C_l>0$ for each $l\geq 3$ such that 
$$\sup_{a\leq t<T_*}\|\pp_t u\|_{H^{l-2}}\leq C_l.$$
By the Fundamental theorem of Calculus, for every $l\geq 3$, $u(t)\to u_*$ in $H^{l-2}(S^1)$ for some $u_*\in H^{l-2}(S^1)$ as $t\uparrow T_*$. 
Hence, $u_*\in C^\infty(S^1)$ and $u(t)\to u_*$ in $C^\infty(S^1)$ as $t\uparrow T_*$. 
By Claim~\ref{claim:trapping}, 
\[
  \|u_*\|_{H^2}\leq\frac{3\rho}{4}<\rho.
\]
Using short-time existence and uniqueness for uniform parabolic PDE again, we can extend the graphical solution $\Gamma_t$ beyond $T_*$ preserving the $H^2$ bound. This is a contradiction to \(T_*=\sup\mathcal T\). Therefore \(T_*=\infty\), proving the proposition.
\end{proof}

\begin{proof}[Proof of Theorem~\ref{thm:main}]
Let $\rho \in (0,\delta)$. By Proposition~\ref{prop:trapping}, there exists $j_0 \in \mathbb{N}$ sufficiently large such that $\Gamma$ is a graphical solution of CSF on $[t_{j_0},\infty)$  and its normal graph function $u$ satisfies
\[
    \|u(\cdot,t)\|_{H^2}<\rho
    \qquad\text{for all }t\geq t_{j_0}.
\]
Fix any integer $m>2$. Lemma~\ref{lem:uniform} gives constants $B_m<\infty$ such that $\|u(\cdot,t)\|_{H^m}\leq B_m$ for $t\geq t_{j_0}.$ By the Sobolev embedding and interpolation inequalities, we have,
\[
\begin{aligned}
\|u(\cdot,t)\|_{C^m}
&\leq C_m\|u(\cdot,t)\|_{H^{m+1}}\\
&\leq
C_m
\|u(\cdot,t)\|_{L^2}^{2/(m+3)}
\|u(\cdot,t)\|_{H^{m+3}}^{(m+1)/(m+3)}\\
&\leq
C_m B_{m+3}^{(m+1)/(m+3)}
\rho^{2 /(m+3)}.
\end{aligned}
\]
for all $t \geq t_{j_0}$. Since $[\Gamma_t] = [\gamma_{u(\cdot, t)}]$ for all $t \geq t_{j_0}$, we have $[\Gamma_t]$ converges to $[\gamma_*]$ in $\Omega$, as $t\to \infty$. If the initial curve $\Gamma_0$ is embedded, then by Lemma~\ref{lem:global}, $\gamma_*$ is an embedded closed geodesic of multiplicity one. This completes the proof. 
\end{proof}

\appendix

\section{Appendix}
\label{sec:appendix}

\setcounter{theorem}{0}
\renewcommand{\thetheorem}{\thesection.\arabic{theorem}}

The following technical lemma is used in Lemma \ref{lem:uniform}. Recall $\mathcal{U}_\delta$ defined in \eqref{eq:U-delta}. 
\begin{lemma}
\label{lem:graph-reparam-compactness}
Let $u_i\in{\mathcal U}_\delta$ and $ h_i \in\Diff^+(S^1)$. Suppose that $u_i\to u_\infty$ in $C^0(S^1)$ and that, for a smooth immersion $\bar{\gamma}\colon S^1\to\Surf$,  $[\gamma_{u_i}] \to [\bar{\gamma}]$ in $\Omega$. 
Then $u_\infty\in\mathcal U_\delta\cap C^\infty(S^1)$ and after passing to a subsequence,
$u_i\to u_\infty$ in $C^\infty(S^1)$. 
\end{lemma}
\begin{proof}
The Sobolev embedding $H^2(S^1)\Subset C^1(S^1)$ shows that \(u_\infty\in\mathcal U_\delta\), $u_i\rightharpoonup u_\infty$ in $H^2(S^1)$, and $u_i\to u_\infty$ in $C^1(S^1)$. By definition of the convergence in $\Omega$, 
there exists $h_i\in \operatorname{Diff}^+(S^1)$ such that $\eta_i:=\gamma_{u_i}\circ h_i\to \bar{\gamma}$ in $C^\infty(S^1,\Surf)$. 
In particular, $|\partial_{\theta}\eta_i| \to |\partial_{\theta}\bar{\gamma}| > 0$. 
By \eqref{eq: bound on Q}, 
\begin{equation}
\label{eqn:h-i-C1-bound-3}
|\partial_\theta\gamma_{u_i}|_g
=
\frac{\ell}{2\pi}Q[u_i]
\in
\left[\frac{\ell}{4\pi},\frac{\ell}{\pi}\right].
\end{equation}
Since \(    h_i'
    =\frac{|\partial_\theta\eta_i|_g}
           {|\partial_\theta\gamma_{u_i}|_g\circ h_i}\), there exist constants $c,C>0$ such that
\begin{equation}
\label{eq:graph-reparam-h-bounds}
c\leq h_i'\leq C
\end{equation}
for all sufficiently large $i$. 
After passing to a subsequence, the Arzel\`a-Ascoli theorem gives \(h_i\to h_\infty\) in \(C^0(S^1)\) for a continuous map $h_\infty:S^1\to S^1$. 
Define \(z_i\colon S^1
\to 
S^1\times(-\varepsilon/2,\varepsilon/2)\) by \(z_i(\theta)
:=
\bigl(h_i(\theta),u_i(h_i(\theta))\bigr).\) Then
\begin{equation}
\label{eq:F-z-i}
F\circ z_i
=
\gamma_{u_i}\circ h_i
=
\eta_i.
\end{equation}
Since $u_i\to u_\infty$ and $h_i\to h_\infty$ in $C^0(S^1)$ uniformly, we obtain $u_i\circ h_i
\to
u_\infty\circ h_\infty$ in $C^0(S^1)$. Consequently,
\begin{equation}
\label{eq:z-i-C0-convergence}
z_i
\to
z_\infty
:=
\bigl(h_\infty,u_\infty\circ h_\infty\bigr)
\qquad\text{in }C^0(S^1).
\end{equation}
Passing to the limit in \eqref{eq:F-z-i} gives $F\circ z_\infty=\bar\gamma.$ 

Fix $\theta_0\in S^1$. Since $F$ is a local diffeomorphism at $z_\infty(\theta_0)$, there exist open neighborhoods
\(z_\infty(\theta_0)\in W_{\theta_0}
\subset
S^1\times(-\varepsilon,\varepsilon)\)
and
\(F(z_\infty(\theta_0))\in V_{\theta_0}\subset\Surf\)
such that
\(F|_{W_{\theta_0}}
\colon
W_{\theta_0}
\to 
V_{\theta_0}\)
is a diffeomorphism. There exists an interval
$I_{\theta_0}\subset S^1$ containing $\theta_0$ such that
\(z_\infty(\overline{I_{\theta_0}})
\subset W_{\theta_0}.\) By \eqref{eq:z-i-C0-convergence},
\(z_i(I_{\theta_0})\subset W_{\theta_0}\) for every sufficiently large $i$. Hence, by \eqref{eq:F-z-i},
\(z_i
=
(F|_{W_{\theta_0}})^{-1}\circ\eta_i\) on 
\(I_{\theta_0}.\) 
Since $\eta_i\to\bar\gamma$ in $C^\infty$, it follows that
\[
z_i
\to 
(F|_{W_{\theta_0}})^{-1}\circ\bar\gamma\equiv z_\infty
\qquad\text{in }C^\infty(I_{\theta_0}).
\]
A finite collection of these intervals covers $S^1$, and therefore $z_i\to z_\infty
$ in $C^\infty(S^1).$ Taking components, 
\begin{equation}
\label{eq:h-and-u-compose-convergence}
h_i\to h_\infty
\quad\text{and}\quad
u_i\circ h_i
\to
u_\infty\circ h_\infty
\qquad\text{in }C^\infty(S^1).
\end{equation}
By \eqref{eq:h-and-u-compose-convergence} and \eqref{eq:graph-reparam-h-bounds}, we obtain that $h_\infty$ is smooth and an orientation-preserving local diffeomorphism.
Since $\operatorname{deg}h_i=1$ for all $i \in \mathbb{N}$, by smooth convergence, $\operatorname{deg}h_\infty=1$. It follows that $h_\infty$ is injective and $h_\infty\in \operatorname{Diff}^+(S^1)$. Hence $h_i^{-1}
    \to 
    h_\infty^{-1}$ in  $C^\infty(S^1,S^1).$ Using  \eqref{eq:h-and-u-compose-convergence} again, we conclude that $u_\infty$ is smooth. Finally, 
$$u_i=(u_i\circ h_i)\circ h_i^{-1}\to (u_\infty\circ h_\infty)\circ h_\infty^{-1}= u_\infty\qquad \text{in }C^\infty(S^1)$$
This completes the proof. 
\end{proof}


\begin{thebibliography}{99}

\bibitem{Angenent1990}
S.~B. Angenent,
\newblock Parabolic equations for curves on surfaces. Part~I. Curves
with $p$-integrable curvature,
\newblock \emph{Ann. of Math. (2)} \textbf{132} (1990), no.~3, 451--483.
\newblock \href{https://doi.org/10.2307/1971426}{doi:10.2307/1971426}.

\bibitem{Angenent1991}
S.~B. Angenent,
\newblock Parabolic equations for curves on surfaces. Part~II.
Intersections, blow-up and generalized solutions,
\newblock \emph{Ann. of Math. (2)} \textbf{133} (1991), no.~1, 171--215.
\newblock \href{https://doi.org/10.2307/2944327}{doi:10.2307/2944327}.

\bibitem{Angenent1999}
S.~Angenent,
\newblock Inflection points, extatic points and curve shortening,
\newblock in \emph{Hamiltonian Systems with Three or More Degrees of
Freedom}, NATO ASI Series~C, vol.~533, Kluwer Academic Publishers,
Dordrecht, 1999, pp.~3--10.
\newblock
\href{https://doi.org/10.1007/978-94-011-4673-9_1}
{doi:10.1007/978-94-011-4673-9\_1}.

\bibitem{Angenent2005}
S.~B. Angenent,
\newblock Curve shortening and the topology of closed geodesics on
surfaces,
\newblock \emph{Ann. of Math. (2)} \textbf{162} (2005), no.~3,
1187--1241.
\newblock
\href{https://doi.org/10.4007/annals.2005.162.1187}
{doi:10.4007/annals.2005.162.1187}.

\bibitem{AryanLee2026}
S.~Aryan and T.-K. Lee,
\newblock Non-uniqueness of geodesic limits and a question of Grayson
and Gage,
\newblock preprint, arXiv:2608.04855v1 (2026).
\newblock \url{https://arxiv.org/abs/2608.04855}.

\bibitem{CerveraMascaroMichor1991}
V.~Cervera, F.~Mascar\'o, and P.~W. Michor,
\newblock The action of the diffeomorphism group on the space of
immersions,
\newblock \emph{Differential Geom. Appl.} \textbf{1} (1991), no.~4,
391--401.
\newblock
\href{https://doi.org/10.1016/0926-2245(91)90015-2}
{doi:10.1016/0926-2245(91)90015-2}.

\bibitem{FeehanMaridakis2020}
P.~M.~N. Feehan and M.~Maridakis,
\newblock \L ojasiewicz--Simon gradient inequalities for analytic and
Morse--Bott functions on Banach spaces,
\newblock \emph{J. Reine Angew. Math.} \textbf{765} (2020), 35--67.
\newblock
\href{https://doi.org/10.1515/crelle-2019-0029}
{doi:10.1515/crelle-2019-0029}.

\bibitem{Gage1990}
M.~E. Gage,
\newblock Curve shortening on surfaces,
\newblock \emph{Ann. Sci. \'Ecole Norm. Sup. (4)} \textbf{23} (1990),
no.~2, 229--256.
\newblock
\href{https://doi.org/10.24033/asens.1603}
{doi:10.24033/asens.1603}.

\bibitem{GageHamilton1986}
M.~E. Gage and R.~S. Hamilton,
\newblock The heat equation shrinking convex plane curves,
\newblock \emph{J. Differential Geom.} \textbf{23} (1986), no.~1,
69--96.
\newblock
\href{https://doi.org/10.4310/jdg/1214439902}
{doi:10.4310/jdg/1214439902}.

\bibitem{Grayson1987}
M.~A. Grayson,
\newblock The heat equation shrinks embedded plane curves to round
points,
\newblock \emph{J. Differential Geom.} \textbf{26} (1987), no.~2,
285--314.
\newblock
\href{https://doi.org/10.4310/jdg/1214441371}
{doi:10.4310/jdg/1214441371}.

\bibitem{Grayson1989}
M.~A. Grayson,
\newblock Shortening embedded curves,
\newblock \emph{Ann. of Math. (2)} \textbf{129} (1989), no.~1, 71--111.
\newblock
\href{https://doi.org/10.2307/1971486}
{doi:10.2307/1971486}.

\bibitem{PludaPozzetta2024}
A.~Pluda and M.~Pozzetta,
\newblock \L ojasiewicz--Simon inequalities for minimal networks:
stability and convergence,
\newblock \emph{Math. Ann.} \textbf{389} (2024), no.~3, 2729--2782.
\newblock
\href{https://doi.org/10.1007/s00208-023-02714-7}
{doi:10.1007/s00208-023-02714-7}.

\bibitem{Pozzetta2022}
M.~Pozzetta,
\newblock Convergence of elastic flows of curves into manifolds,
\newblock \emph{Nonlinear Anal.} \textbf{214} (2022), Paper
No.~112581.
\newblock
\href{https://doi.org/10.1016/j.na.2021.112581}
{doi:10.1016/j.na.2021.112581}.

\bibitem{Simon1983}
L.~Simon,
\newblock Asymptotics for a class of nonlinear evolution equations,
with applications to geometric problems,
\newblock \emph{Ann. of Math. (2)} \textbf{118} (1983), no.~3, 525--571.
\newblock
\href{https://doi.org/10.2307/2006981}
{doi:10.2307/2006981}.

\bibitem{sun2024curve}
Q.~Sun,
\newblock Curve shortening flow of space curves with convex projections,
\newblock preprint, arXiv:2410.08399 (2024).
\newblock \url{https://arxiv.org/abs/2410.08399}.

\bibitem{sun2025huisken}
Q.~Sun,
\newblock Huisken's distance comparison principle in higher codimension,
\newblock preprint, arXiv:2509.16823 (2025).
\newblock \url{https://arxiv.org/abs/2509.16823}.

\bibitem{sun2025singularities}
Q.~Sun,
\newblock Singularities of curve shortening flow with convex projections,
\newblock preprint, arXiv:2510.14863 (2025).
\newblock \url{https://arxiv.org/abs/2510.14863}.
\bibitem{Edelen2015}
N.~Edelen,
\newblock Noncollapsing of curve-shortening flow in surfaces,
\newblock \emph{Int. Math. Res. Not. IMRN} 2015, no.~20,
10143--10153.
\newblock
\href{https://doi.org/10.1093/imrn/rnu271}
{doi:10.1093/imrn/rnu271}.

\bibitem{Oaks1994}
J.~A. Oaks,
\newblock Singularities and self-intersections of curves evolving on
surfaces,
\newblock \emph{Indiana Univ. Math. J.} \textbf{43} (1994), no.~3,
959--981.
\newblock
\href{https://doi.org/10.1512/iumj.1994.43.43042}
{doi:10.1512/iumj.1994.43.43042}.

\bibitem{Zhu1998}
X.-P. Zhu,
\newblock Asymptotic behavior of anisotropic curve flows,
\newblock \emph{J. Differential Geom.} \textbf{48} (1998), no.~2,\
225--274.
\newblock
\href{https://doi.org/10.4310/jdg/1214460796}{doi:10.4310/jdg/1214460796}.

\bibitem{JohnsonMuraleetharan2010}
D. L. Johnson and M. Muraleetharan,
\newblock Singularity formation of embedded curves evolving on surfaces
by curvature flow,
\newblock \emph{Int. J. Pure Appl. Math.} \textbf{61} (2010), no.~2,\
121--146.
\newblock
\href{https://www.lehigh.edu/\~dlj0/papers/paper-short-hamilton.pdf}
{Author PDF}.

\bibitem{Evans1998}
L.~C. Evans,
\newblock \emph{Partial Differential Equations},
\newblock Graduate Studies in Mathematics, vol.~19,
\newblock American Mathematical Society, Providence, RI, 1998.
\newblock \href{https://doi.org/10.1090/gsm/019}{doi:10.1090/gsm/019}.


\bibitem{schulze2014uniqueness}
F.~Schulze,
\newblock Uniqueness of compact tangent flows in mean curvature flow,
\newblock \emph{J. Reine Angew. Math.} \textbf{690} (2014), 163--172.
\newblock
\href{https://doi.org/10.1515/crelle-2012-0070}
{doi:10.1515/crelle-2012-0070}.

\bibitem{colding2015uniqueness}
T.~H. Colding and W.~P. Minicozzi II,
\newblock Uniqueness of blowups and \L ojasiewicz inequalities,
\newblock \emph{Ann. of Math. (2)} \textbf{182} (2015), no.~1,
221--285.
\newblock
\href{https://doi.org/10.4007/annals.2015.182.1.5}
{doi:10.4007/annals.2015.182.1.5}.

\bibitem{chodosh2021uniqueness}
O.~Chodosh and F.~Schulze,
\newblock Uniqueness of asymptotically conical tangent flows,
\newblock \emph{Duke Math. J.} \textbf{170} (2021), no.~16,
3601--3657.
\newblock
\href{https://doi.org/10.1215/00127094-2020-0098}
{doi:10.1215/00127094-2020-0098}.

\bibitem{zhu2020ojasiewicz}
J.~J. Zhu,
\newblock \L ojasiewicz inequalities, uniqueness and rigidity for
cylindrical self-shrinkers,
\newblock \emph{Camb. J. Math.} \textbf{13} (2025), no.~1, 173--224.
\newblock
\href{https://doi.org/10.4310/CJM.250319033143}
{doi:10.4310/CJM.250319033143}.

\bibitem{lee2024uniqueness}
T.-K. Lee and X.~Zhao,
\newblock Uniqueness of conical singularities for mean curvature flows,
\newblock \emph{J. Funct. Anal.} \textbf{286} (2024), no.~1, Paper
No.~110200, 23~pp.
\newblock
\href{https://doi.org/10.1016/j.jfa.2023.110200}
{doi:10.1016/j.jfa.2023.110200}.

\bibitem{bryan2026sharp}
P.~Bryan, M.~Langford, and J.~J. Zhu,
\newblock Sharp distance comparison for curve shortening flow on the round
sphere,
\newblock \emph{Ann. Sc. Norm. Super. Pisa Cl. Sci. (5)} (2026),
published online, 12~pp.
\newblock
\href{https://doi.org/10.2422/2036-2145.202501_010}
{doi:10.2422/2036-2145.202501\_010}.



\end{thebibliography}
\end{document}